\documentclass[12pt]{amsart}

\usepackage{amsmath,amssymb,lscape, comment, tikz}
\usepackage{filecontents}
\usepackage{fullpage}
\usepackage[all]{xy}

\newtheorem{theorem}{Theorem}[section]
\newtheorem{proposition}[theorem]{Proposition}
\newtheorem{lemma}[theorem]{Lemma}
\newtheorem{cor}[theorem]{Corollary}

\theoremstyle{definition}
\newtheorem{definition}[theorem]{Definition}
\newtheorem{example}[theorem]{Example}

\theoremstyle{remark}
\newtheorem{remark}[theorem]{Remark}

\numberwithin{equation}{section}

\def\DJ{{\hbox{D\kern-.8em\raise.15ex\hbox{--}\kern.35em}}}
\def\DJo{$\;$\kern-.4em
    \hbox{D\kern-.8em\raise.15ex\hbox{--}\kern.35em koori\'c}}

\def\al{{\alpha}}

\def\ga{{\gamma}}
\def\d{{\delta}}
\def\vf{{\varphi}}
\def\la{{\lambda}}
\def\na{{\nabla}}
\def\om{{\omega}}
\def\Om{{\Omega}}
\def\si{{\sigma}}
\def\Si{{\Sigma}}

\def\bZ{{\mathbb {Z}}}

\def\bR{{\mathbb {R}}}
\def\bC{{\mathbb {C}}}
\def\g{{\mathfrak{g}}}

\def\pA{{\mathcal A}}

\def\pF{{\mathcal F}}
\def\pE{{\mathcal E}}

\def\pL{{\mathcal L}}

\def\pP{{\mathcal P}}

\def\pO{{\mathcal O}}
\def\ad{{\rm ad}}
\def\tr{{\rm tr}}

\def\Vol{{\mbox{\rm Vol}}}

\def\ad{{\mbox{\rm ad}}}
\def\Ad{{\mbox{\rm Ad}}}
\def\Hom{{\mbox{\rm Hom}}}
\def\Lie{{\mbox{\rm Lie}}}

\def\Aut{{\mbox{\rm Aut}}}

\def\diag{{\mbox{\rm diag}}}
\def\GL{{\mbox{\rm GL}}}
\def\SL{{\mbox{\rm SL}}}

\def\U{{\mbox{\rm U}}}
\def\Sp{{\mbox{\rm Sp}}}
\def\rk{{\mbox{\rm rk\,}}}

\def\h{{\mathfrak{h}}}
\def\u{{\mathfrak{u}}}
\def\e{{\epsilon}}
\def\pZ{{\mathcal Z}}
\def\pM{{\mathcal M}}
\def\pH{{\mathcal H}}

\def\ss{{\subset}}
\def\sseq{{\subseteq}}

\newcommand{\mat}[1]{\left(\begin{matrix}#1 \end{matrix}\right)}
\newcommand{\smat}[1]{\left(\begin{smallmatrix}#1\end{smallmatrix}\right)}

\begin{document}

\title[$G$ monopoles]
{Singular $G$-monopoles on $S^1\times \Si$}

\author[B.H. Smith]
{Benjamin H. Smith}

\address{Department of Mathematics \& Statistics, McGill University,
Montr\'{e}al, QC, 
%H3A OB9, 
Canada}

 \email{bh2smith@gmail.com}

\thanks{}


\date{}

\begin{abstract}
This article provides an account of the functorial correspondence between irreducible singular $G$-monopoles on $S^1\times \Si$ and $\vec{t}$-stable meromorphic pairs on $\Si$.  The main theorem of \cite{CH} is thus generalized here from unitary to arbitrary compact, connected gauge groups. The required distinctions and similarities for unitary versus arbitrary gauge are clearly outlined and many parallels are drawn for easy transition. Once the correspondence theorem is complete, the spectral decomposition is addressed.
\end{abstract}
\maketitle 

%\tableofcontents
%\subjclassname {}

\section{Introduction}

The main goal here is to provide a proof of the bijective Kobayashi-Hitchin type correspondence between the moduli space of \emph{singular $G$-monopoles} over $S^1\times \Sigma$ and the space of \emph{$\vec{t}$-polystable meromorphic pairs} $(\mathcal{P}, \rho)$. Since complex vector bundles are equivalent to principal $\GL_n(\bC)$-bundles, the results of  \cite{CH} form a model for the constructions and results found here. In this setting, however, we will not have the luxury of working with the Lie algebra of skew-hermitian matrices, which form an inductive system. Careful considerations will be made about the properties of the more general Lie algebras involved. For this reason, $G^c$ will denote a complex reductive Lie group (realizable as the complexification of a compact, connected real reductive Lie group $G$). 

The main theorem, stated in full generality, is provided as follows;

\begin{theorem}\label{main}
There is a bijective correspondence between the moduli space \[\pM^{irr}_{k_0}(G,S^1\times \Sigma, \{(p_i, \mu_i)\}_{i=1}^N)\] of irreducible principal $G$-monopoles over $S^1\times \Sigma$ with singularities at $p_i\in S^1\times \Si$ of $\mu_i$-Dirac type, having degree $k_0$ over $\{0\}\times\Si$ and the moduli space 
\[\pM_{\vec{t}s}(\Sigma, {\textbf K}, k_0)\]
 of $\vec{t}$-stable meromorphic pairs $(P, \psi)$, where $P$ is a holomorphic principal $G$-bundle of degree $k_0$ over $\Sigma$ and $\psi$ is a meromorphic section of $\Aut_G(P)$ taking the form 
\[F_i(z) \mu_i(z - z_i) G_i(z)\] 
when expressed locally near $z_i$ with $F_i, G_i$ holomorphic-invertible and $\mu_i$ a cocharacter of the complexified gauge group $G^c$.
\end{theorem}

In less cryptic terminology, this theorem states that one may parameterize the moduli space of $G$-monopoles over $S^1\times \Sigma$ having singularities of Dirac-type by the more tractable complex algebraic moduli space of $\vec{t}$-stable meromorphic pairs. There is a family of these moduli spaces, parameterized by the location of singularities on $\Sigma$ and indexed by the combinatorial data given by the initial degree$k_0$, and ``charge'' $\mu_i$ of the bundle at the singularities. That is to say, the real work lies in verifying that a $\vec{t}$-stable meromorphic pair $(P,\rho)$ is the skeletal information required to uniquely construct a solution to the monopole equation. 

The method used to reconstruct a monopole from its singular data is an interesting application of heat flow on the space of positive hermitian metrics which, in the course of doing so, makes use of the celebrated Hopf-fibration. Heuristically, one wishes to, holomorphically, patch together a $G$-bundle on $S^1\times \Sigma$ having the correct prescribed 'twisting', so to be in the correct topological isomorphism class. This is done by patching together a metric (using a partition of unity) that will be a parametrix of the solution having the correct singular data. Once this metric is defined, the heat flow is employed to evenly distribute the curvature, induced by the metric, towards a solution to the monopole equation.

Historically there have been several results involving classifications of these types and the general picture is known as the \emph{Kobayashi-Hitchin correspondence}. There are three foundational works in this area; namely the papers of Donaldson \cite{Do83, Do85} and Uhlenbeck-Yau \cite{UY86, UY89} in establishing the Kobayashi-Hitchin correspondence for holomorphic vector bundles on compact K\"{a}hler manifolds \cite{LT}. The progression of these results is the work of many mathematicians starting with Narasimhan-Seshardi \cite{NS} for Riemann surfaces, Donaldson \cite{Do83, Do85,Do87} again for Riemann surfaces and also algebraic surfaces, and Uhlenbeck-Yau \cite{UY86,UY89} for compact K\"{a}hler manifolds. A careful analysis of heat flow in these settings, and more generally in situations with singularities, is due to Simpson \cite{S}.  A good reference for the completed Kobayashi-Hitchin correspondence was presented by L\'{u}bke and Teleman \cite{LT} in great detail and generality. 

In our situation, the solutions to the Bogomolny (monopole) equation are required to have singularities. In 1988, Simpson \cite{S} provided a short list of assumptions sufficient to guarantee the required long term existence of the heat equation in these cases. Our domains and initial conditions fit Simpson's profile (as first employed in \cite{CH}) and so we have the existence of our solutions with the exception of singular neighbourhoods that must be considered separately.

It was M. Pauly \cite{MP}, following unpublished work of Kronheimer who first dealt with Dirac-type singular monopoles on 3-balls. He displayed, via a radially extended version of the Hopf fibration, a correspondence between Dirac-type monopoles on $B^3\backslash \{0\}$ and smooth $S^1$-invariant anti self-dual connections on $B^4\backslash \{0\}$. This was used to solve the problem of classifying singular Hermitian-Einstein (i.e. $G= \U(n)$) monopoles on $S^1\times \Sigma$ which was recently worked out by B. Charbonneau and J. Hurtubise \cite{CH}.

Section 2 provides background on the Bogomolny equation and the $\mu$-Dirac monopole in the context of principal bundles. The moduli spaces and characteristic classes of interest are defined and partially analyzed in Section 3. Section 4 is devoted to the stability theory of monopoles and meromorphic pairs. Proof of the main Theorem \ref{Main} is found in Section 5. Finally, in the last section, the abelianization (or spectral decomposition) of our monopoles is provided along with some examples of Weyl-invariant compactifications of maximal tori. 

\section{Background and basic objects}

Throughout this paper, denote by $G^c$ a complex reductive Lie group of rank $n$, its maximal compact subgroup $G$,  a Riemann surface $\Sigma$ with Hermitian metric, a circle $S^1$ of circumference $\tau$ with standard metric and impose the product metric on the manifold $S^1\times \Sigma$ having coordinates $t,z = x+iy$.

\subsection{Bogomolny equations and generalizations}
  Let $P$ be a principal $G^c$-bundle on $S^1\times \Si$,
\[Y:=S^1\times \Sigma \backslash \{p_1,\dots, p_N\}\] 
where each $p_i$ has coordinates $(t_i, z_i)\in S^1\times \Sigma$ and, purely for the sake of notational convenience, the $t_i$'s and $z_i$'s are assumed to be distinct. The restriction of $P$ to sufficiently small spheres about each $p_i$ comes with a reduction to the maximal real torus $T\subset G$ whose transition function on the 2-sphere is given by some cocharacter $\mu_i$ of $T$. Suppose that $P$ admits a $G$-connection $\na$ and a section $\Phi \in H^0(Y, \ad (P))$, of the adjoint bundle called a \emph{Higgs field}. The triple $(P, \na, \Phi)$ satisfies the \emph{Bogomolny equation} if
\begin{equation}\label{BE}
F_\na = *d_\na \Phi.
\end{equation}
It can be shown that this equation is equivalent to a special case of a reduction from the anti self-dual (ASD) equations over $S^1\times Y$. 

%This reduction and many other simple facts in this preliminary section can be found in \cite{CH}. Since the purpose here is to extend the main result of \cite{CH} from the vector bundle setting to principal bundles with reductive structure group $G$, the reader is referred to this source for many technicalities.

Unfortunately equation \ref{BE} imposes unnecessarily strong constraints on the first Chern classes (i.e. that they average to zero in a suitable sense) so the following, slightly weaker, form will be considered here to allow for solutions with arbitrary degree. That is to say, the triple $(P,\na, \Phi)$ is said to satisfy the \emph{Hermitian-Einstein-Bogomolny (HEB) equation} if
\begin{equation}\label{HEB}
F_\na - i C  \cdot\om_\Sigma = *d_\na \Phi
\end{equation}
where $C$ is in the center, $\pZ(\g)$, of the Lie algebra $\mathfrak{g} = Lie(G)$ and $\om_\Sigma \in \Omega^2(\Sigma)$ represents the K\"{a}hler form of our Riemann surface. The difference here between equations \eqref{BE} and \eqref{HEB} is an extra term which allows for non-zero global central curvature. Note that central elements of $\mathfrak{g}$ are invariant under conjugation and thus may be equivalently viewed as sections of $\ad(P)$. 

Since our domain a product manifold, equation \eqref{HEB} can be split into components as stated in the following lemma.

\begin{lemma}\label{breakdown}
The HEB-equation \eqref{HEB} can be re-expressed as the following three equations;
\begin{equation}\label{HEB1}
F_\Sigma- \na_t\Phi = iC,
\end{equation}
\begin{equation}\label{HEB2}
[\na^{0,1}_\Sigma, \na_t-i\Phi]=0
\end{equation}
\begin{equation}\label{HEB3}
[\na^{1,0}_\Sigma, \na_t+i\Phi]=0
\end{equation}
where $F_\Sigma$ is the surface component of the curvature tensor (i.e. $F = F_\Sigma\om_\Si + \cdots$) and $\na = \na_\Sigma^{0,1} d\bar{z} + \na_\Sigma^{1,0}dz + \na_tdt.$ Note that the third equation is merely the dual of the second.
\end{lemma}

\begin{proof}
This is shown by breaking equation \eqref{HEB} into components and remembering that it is ``unitary'' (in the $G$-sense). Extracting the surface component, $\Sigma = dx\wedge dy$, of \eqref{HEB} gives 
\[F_\Sigma - iC  = \na_t\Phi\]
where the Hodge-star on the right hand side of \eqref{HEB} takes surface components to time, $\langle t\rangle $, components and vice-versa. 

For equation \eqref{HEB2}, extract components $\langle x,t\rangle$, $\langle y,t\rangle$ and combine them. On the left hand side, the $\langle x,t\rangle $ component of curvature is realized as the commutator $[\na_x,\na_t]$ which gives the equation 
\[[\na_x,\na_t] - 0 = -\na_y\Phi = -[\na_y, \Phi]\] 
where the negative is recognized as coming from the Hodge-star applied to the ordered basis $\{x,y,t\}$. Similarly, the $\langle y,t\rangle $ component is
\[[\na_y, \na_t] = \na_x\Phi = [\na_x,\Phi].\]
Multiplying the second by $i$ and adding these together gives
\[[\na_x+i\na_y,\na_t] = [-\na_y+i\na_x, \Phi] = [\na_x+ i \na_y, i \Phi]\]
and simplification of this is precisely equation \eqref{HEB2}.
\end{proof}

\subsection{The $\mu$-Dirac monopole}\label{Dirac}
This section is based on standard knowledge of complex line bundles on $S^2$. 
%The $\U(1)$-Dirac monopole shall be referred to frequently and built upon within this document, so the reader is referred to \cite{CH} for necessary background.
Throughout the remainder of this article, let $\mu\in X_*(T) = \Hom(S^1,T)$ be a cocharacter of a fixed maximal torus $T\ss G$. 

\begin{definition}
For any real compact torus $T$, a \emph{$\mu$-Dirac monopole} is a principal $T$-bundle over $\bR^3\backslash\{0\}$ of degree $\mu$, equipped with a connection $\na$ and Higgs field $\phi$ satisfying the Hermitian-Einstein-Bogomolny equation \eqref{HEB} provided as follows:

On $\bR^3$, one has spherical coordinates related to Euclidean by \[(t, x,y) = (R\cos \theta, R\cos \psi\sin \theta, R\sin \psi\sin \theta)\] and volume form 
\[dV = R^2\sin \theta dR d\theta d\psi = - r^2 dr d(\cos \theta d\psi). \]

For any $\mu\in X_*(T)$ define the principal $T$-bundle $L_\mu$ over $\bR^3\backslash\{0\}$ by the transition function $g_{\pm} = \mu(\psi)$ between neighbourhoods \[U_\pm = \bR^3\backslash \{\pm t\geq 0\}.\] 
Any section on this bundle may be expressed by maps $\sigma_\pm:U_\pm\to T$ satisfying $\sigma_- = g_{\pm}\sigma_+$.

Now, consider a connection defined locally by Lie-algebra-valued 1-forms
\[A_\pm = \frac{i\mu_*}{2}(\pm 1+\cos \theta)d\psi\]
where $\mu_*\in \Lie(T)$ is the differential of $\mu$ evaluated at 0 and the Higgs field $\phi = \frac{i\mu_*}{2R}$. It is clear that 
\[\na\phi = d\phi +[A,\phi] = d\phi = -\frac{i\mu_*}{2R^2}dR = * \left(\frac{i\mu_*}{2}d(\cos\theta d\psi)\right) =*F_\na, \]
so that the pair $(\na, \phi)$ satisfies the Bogomolny equation \eqref{BE} and, equivalently, equation \eqref{HEB} with $C = 0$.
\end{definition}

If $U_\pm$ represents the open cover of $\bR^3\backslash\{0\}$ obtained by removing the positive/negative $z$-axes, then the overlap $U_+\cap U_-$ is homotopy-equivalent to a circle and so the transition functions defining such a bundle can be given, up to homotopy, by a cocharacter $\mu\in X_*(T)$ and sections $\sigma$ are uniquely expressed as maps $\si_\pm:U_\pm \to T$ satisfying $\si_+ = \mu \cdot \si_-$. 

Following this, one has

\begin{lemma}
The $\mu$-Dirac monopoles are all induced from the standard $S^1$-Dirac monopole by the cocharacter $\mu\in X_*(T)$.  
\end{lemma}

\begin{proof}

First note that, as for any bundle over a sphere, the smooth isomorphism class of any torus bundle is determined by the homotopy classes of maps $[S^1, T]$ for which one may choose a cocharacter $\mu\in X_*(T)$ as a representative. Thus this torus bundle is isomorphic to the $T$-bundle induced by $\mu$ from the line bundle, $L_1$ of charge 1 over $\bR^3\backslash \{0\}$. That is, one may consider bundles of the form
\[L_1(\mu) := L_1\times_{S^1} T\]
where the diagonal action of $S^1$ on $L_1$ is as usual and via $\mu$ on $T$. 

Having that any $T$-bundle on $\bR^3\backslash \{0\}$ realized as $L_1(\mu)$ for some cocharacter $\mu\in X_*(T)$, it is natural to choose the necessary connection and Higgs field to be obtained through $\mu$ as well. Indeed, with connection form defined locally on the open cover $U_\pm:= \bR^3\backslash \{\mp z \geq 0\}$ as  $\om_\pm = \mu_*(A_\pm)$ and Higgs field $\Phi := \mu_*(\phi)$ where $A$ and $\phi$ are the connection and Higgs field for the model Dirac monopole of charge 1, defined in \cite{CH}. It is then tautological to verify that $(L_1(\chi), \om, \Phi)$ satisfies the monopole equation. 
\end{proof}

With this identification, there is no need to pursue the structure of the $\mu$-Dirac monopole further. Calculations for the change between holomorphic and unitary gauges are the same as for vector bundles (c.f. \cite{CH})

\section{Singular $G$-monopoles, holomorphic structures and meromorphic pairs}

This section introduces and elaborates on the analytic and topological details involving both singular $G$-monopoles on $S^1\times \Si$ and their eventual algebraic equivalent, meromorphic pairs. The stability of both is discussed in depth including motivation and consistency arguments from the standard theory. 

A map, $\pH$, from singular monopoles to meromorphic pairs is defined and shown to preserve stability. This was proven for singular Hermitian-Einstein ($\U_n$) monopoles in \cite{CH}. However, their proof relies on an inductive argument on the rank of the group and does not carry over to arbitrary reductive gauge (e.g. the exceptional Lie group $G_2$ does not admit an inductive system). Here we adapt from similar proofs found in \cite{K, LT} and heavily rely on the fact that, loosely stated, the curvature of holomorphic subbundles is boundeded by the total curvature. This is the essential idea used in the proof of the Kobayashi-Hitchin correspondence, but here the argument is adapted for meromorphic Chern forms.

\subsection{Singular $G$-monopoles}\label{sgms}

For a point $p$ in a three manifold $Y$, let $R$ represent the geodesic distance to $p$ and use a normal coordinate system $(t,x,y)$ centred at $p$ for which the metric in these coordinates is represented by $I + \pO(R)$ as $R\to 0$.  Let $(\theta, \psi)$ represent angular coordinates, as above, for the $\mu$-Dirac monopole on the sphere of constant radius $R = c$ and denote the open ball defined by $R<c$ by $B^3$.

\begin{definition}
A solution $(P, \na,\Phi)$ to the HEB equation \eqref{HEB} on $Y \backslash \{p\}$ has a \emph{singularity of $\mu$-Dirac type} at $p$ if:
\begin{itemize}
\item locally, on $B^3\backslash \{p\}$, $P$ admits a reduction of structure group to $T$ which is $G$-isomorphic (replacing \emph{unitarily isomorphic}) to the $\mu$-Dirac monopole $T_\mu$, and
\item under this isomorphism, in the two open sets, $U_\pm = \bR^3\backslash \{\pm t\geq 0\}$, trivializing $P$ on $B^3$ induced by standard trivializations of the $T_\mu$ (so that the $P$-trivializations have transition function given by $\mu$), one has, in both trivializations, that\footnote{Note here that $\mu_*= \frac{d\mu}{d\psi}|_{\psi = 0}$ is intended to mimic the formulation in $\GL_n$ which reads \[i\diag(k_1, \dots,k_n) = \frac{d}{d\psi}|_{\psi = 0} \diag(e^{ik_1\psi},\dots,e^{ik_n\psi})\]}
\[\Phi = \frac{\mu_*}{2R}+ \pO(1) \text{ and } \na(R\Phi) = \pO(1)\]
\end{itemize}
Furthermore, a solution to equation \eqref{HEB} with singularities $\{p_j\}_{j=1}^N$ of $\mu_j$-Dirac type a is called a \emph{singular $G$-monopole} (of Dirac-type).
\end{definition}

\begin{remark}
Heuristically, this definition says that a solution with singularity of Dirac type is locally (in a neighbourhood of a singular point) comparable to a $\mu$-Dirac monopole. % From the perspective of bundle construction via sheaf cohomology, any section $\si \in \Gamma(P)$ locally takes values in the maximal torus $T$ of $G$.

The second part of the definition ensures, first that the Higgs field respects the local decomposition of $P$ into Dirac monopoles and the second constraint, via equation \eqref{HEB}, ensures that the curvature is $\pO(R^{-2})$ and hence integrable in neighbourhoods of singularities. Indeed, 
\[\pO(1) = \na(R\Phi) = dR\wedge \Phi + R\cdot d_\na \Phi = dR \wedge \Phi + R\cdot (*F_\na - *i CI_n \cdot \om_{\Sigma})\]
implying 
\[*F_\na = \frac{1}{R}(\pO(1) - dR\wedge \Phi ) + *i C \cdot \om_{\Sigma} = \frac{\pO(1) + \pO(R^{-1})}{\pO(R)} + \pO(1) = \pO(R^{-2}) \] 
\end{remark}

The \emph{moduli space of irreducible singular $G$-monopoles}\footnote{defined here simply as a set} on $S^1\times \Sigma$ having Dirac singularities of type $\mu_j$ at $p_j = (t_j, z_j)$ for $j = 1, \dots, N$ is denoted by
\[\pM_{k_0}^{irr}(G, S^1\times \Sigma, \{(p_j, \mu_j)\}_{j=1}^N).\]

\subsection{Holomorphic structures and scattering}\label{scattering}

A \emph{holomorphic structure} on $Y$, will be an intermediary object, obtained by complexification of $P$, when passing from monopoles to meromorphic pairs. However, such objects can be defined independently from those obtained through monopoles. 

\begin{definition}\label{holstr}
A \emph{holomorphic structure} on a $G^c$-bundle $P^c$ over $Y$ is defined by two commuting, covariant (local) differential operators 
\[\na_\Si^{0,1}:\Gamma_\ell(P) \to \Gamma_\ell(P)\otimes (T\Si^{0,1})^*\text{ and } \na_t^c:\Gamma_\ell(P)\to \Gamma_\ell(P)\]
expressed locally as 
\[(\bar{\partial_z} + A^{0,1}_\Si)d\bar{z} \text{ and } \partial_t - i \vf\]
such that near singularities there exists a reduction to $G$ and $\na_t^c$ has the asymptotics of a Dirac-singularity.
\end{definition}

This definition allows for a tangible notion of holomorphic sections over an odd-dimensional domain.
 
\begin{definition}
A (local) section $\si\in \Gamma_\ell(P^c)$ is \emph{holomorphic} if it is parallel with respect to both $\na_\Si^{0,1}$ and $\na_t^c$. That is, $\si$ is holomorphic in the usual sense when restricted to any complex slice $\Si_t$, and satisfies $\na_t^c\si = 0$ (i.e. respecting the commutative nature of the operators).
\end{definition}

One sees, via Equation \eqref{HEB2} in Lemma \ref{breakdown}, that the complexification of a monopole $(P, \na, \Phi)$ admits a holomorphic structure. Concretely,

\begin{proposition}
There exists a forgetful map from monopoles to holomorphic structures on $Y$ given by 
\[(P, \na, \Phi)\mapsto (P^c,\na^{0,1}_\Si, \na^c)\]
where $\na^{0,1}_\Si =  \na_{|\{0\}\times \Si}^{0,1}$ and $\na^c = \na_t - i \Phi$.
\end{proposition}

To holomorphic structures, one may apply the following scattering technique. The \emph{scattering operator} is the second differential operator, $\na^c$, of a holomorphic structure (also, found as the second term in the commutator from equation \eqref{HEB2}). This is a linear first order differential operator in the $S^1$-direction of $S^1\times \Sigma$ and amounts to a complex parallel transport\footnote{Indeed, when the Higgs field is zero, this is exactly the parallel transport in the $t$-direction.} when applied to sections. That is, setting $P^c := P\times_G G^c$ (i.e. the complexification of $P$) let parallel sections $\si \in \Gamma_\ell (P^c)$ satisfy
\[\na^c\si = 0.\]

As usual, whenever the curve $[t,t']\times \{z\}$ contain no singularities, this provides a smooth, fibre-wise isomorphism, 
\[\rho_{t,t'}:P^c_{(t, z)}\to P^c_{(t',z)}\]
defined more precisely as; For each $p\in P^c_{(t,z)}$, let $\ga$ be the unique solution to $\na^c\ga = 0$ with $\ga(t) = p$. Then $\rho_{t,t'}(g) = \ga(t')$.

For intervals $[t, t']$ containing no singularities, integration of the scattering operator defines an isomorphism between $P_{\{t\}\times \Sigma}$ and $P_{\{t'\}\times \Sigma'}$. When there is a singularity at some time $t_i \in (t,t')$ consider, for simplicity, the singularity at the origin of a chart for $\Sigma$ with time considerations as $-1<0<1$.  The result (\cite{CH} Proposition 2.5) is that 

\begin{proposition}\label{scatmap}
In holomorphic trivializations at $t = \pm 1$ the scattering map $\rho_{-1,1}$ is locally expressed in the form 
\[h(z)\mu(z)g(z)\] 
with $h,g:U\subset \bC\to G$ holomorphic and $\mu:\bC^* \to T^c$ is a map into a maximal torus of $G$. Note that the coordinate $z$ has been chosen so that the singularity is at $0$.

We say that a map $\rho:U\to G$ admitting this type of local decomposition is \emph{encoded by $\mu$ at $z$}.
\end{proposition}

To see the result in the principal bundle setting, note that by \cite{CH}, it holds in any representation of $G$. 

\subsection{Meromorphic pairs}

As vaguely described in the statement of Theorem \ref{main}, a meromorphic pair $(\pP, \rho)$ is a holomorphic principal $G$-bundle $\pP$ over a Riemann surface $\Sigma$ and $\rho \in \pM(\Aut(P))$ is a section of $\Aut(P)$ which is meromorphic over $\Si$. More concretely, 

\begin{definition}
A \emph{meromorphic pair} of type $(\vec{\mu},\vec{z}) = \{(\mu_1, z_1),\dots, (\mu_N, z_N) \}$ is a pair $(P, \rho)$ where $P$ is a holomorphic principal $G$-bundle on $\Sigma$ and $\rho \in \pM(\Aut(P))$ is a meromorphic automorphism of $P$ whose singular data is encoded by the cocharacter $\mu_j$ at $z_j\in \Si$. So then \\$\rho:P\to P$ is an automorphism of $P$ on the Zariski-open neighbourhood $\Sigma\backslash \{z_1,\dots, z_N\}$.
\end{definition}

An example of such objects is achieved when considering the forgetful map which takes the holomorphic structure of a singular $G$-monopole $(P, \na, \Phi)$ to $(P^c_t, \rho_{t,t+\tau})$ where $P^c_t:= P^c_{|\{t\}\times \Sigma}$ is the restriction of the complexified bundle $P^c$ on $S^1\times \Sigma$ to some non-singular time $t\in S^1$ and $\rho_{t,t+\tau}$ the monodromy obtained from scattering along $S^1$ with $\na^c = \na_t- i \Phi$. 

Thus,

\begin{proposition}
Every holomorphic structure $(P^c, \na^{0,1}_\Si, \na^c)$ on $Y$ gives rise to a meromorphic pair $(\pP, \rho)$ by restriction of $P^c$ to any non-singular slice $\{t\}\times \Si$ and the monodromy obtained by integrating the scattering operator $\na^c$ around the circle. 
\end{proposition}

The \emph{moduli space of meromorphic pairs} over $\Sigma$ of degree $k_0$ and singular type $\textbf{K} = \{(\mu_j, z_j)\}_{j=1}^N$ will be denoted by
\[\pM(\Sigma, k_0, \textbf{K}).\]

\subsection{From singular monopole to meromorphic pair}

%In summary, now that the objects of importance are well-defined and familiar, 
Define a forgetful map as the composition of maps from monopoles to holomorphic structures and finally to meromorphic pairs
\[\pH:\pM_{k_0}^{irr}(G, S^1\times \Sigma, \{(p_i, \mu_i)\}_{i=1}^N) \to \pM(\Sigma, \textbf{K}) \]
as
\[\pH(P, \na, \Phi) := (P^c_{|\{0\}\times \Si}, \rho_{0,\tau})\]
where $P^c_{|\{0\}\times \Si}$ is the restriction of the complexification $P^c \to Y$ to the slice $\{0\}\times\Sigma$ (note $t=0$ is assumed to be a non-singular time) and $\rho_{0,\tau}$ is the meromorphic automorphism of $P_0$ resulting from the monodromy by scattering all the way around the circumference $S^1$. 

First note that the $P^c_{|\{0\}\times \Si}$ component in the image of $\pH$ is a \emph{holomorphic principal $G$-bundle} over $\Sigma$ because the slice $\{0\}\times \Sigma$ of $S^1\times \Sigma$ has been chosen so not to contain any singular points. Also, since $P^c_{|\{0\}\times \Si}$ is the restriction of a monopole, it is furthermore already equipped with the holomorphic differential $\na^{0,1}_\Sigma$ (as shown by Lemma \ref{HEB2}).

\subsection{The topology and degree of a $G$-bundle on $Y$}

The topological classification for principal $G$-bundles over a fixed base manifold $Y$ is given by homotopy classes of  maps $[Y; BG]$ where $BG$ is the classifying space of $G$. In our case, the base manifold $Y$ is the complement of a finite collection of points in a compact 3-manifold. Thus $Y$ deformation retracts (i.e. is homotopic) to a 2-dimensional CW-complex having $(N+1)$ cells in dimension 2 (namely $Y \simeq Y^1 \cup Y^2$ is the skeletal decomposition where $Y^2 = \Sigma \cup \left(\bigcup_{i=1}^N S^2_i\right)$ . In fact, since there are $N$ punctures in $Y$, the integer second homology is $H_2(Y; \bZ) \cong \bZ^{N+1}$.

With $G$, a compact, connected real algebraic group one finds that $0 = \pi_0(G) = \pi_1(BG)$ which
implies
\[\pi_1(G) = \pi_2(BG) \cong H_2(BG)\]
where the last equivalence is due to Hurewicz's Theorem since $\pi_1(BG) = 0$. Thus, classification of $G$-bundles on $Y$ amounts to the classification of the bundles on a bouquet of $(N+1)$ 2-spheres since the 1-skeleton contracts to a point after mapping to $BG$. 

Considering the characteristic classes obtained by pullback from \\$H^2(BG)$, one has (by the Universal Coefficient Theorem and Hurewicz's Theorem respectively) that 
\[H^2(BG, \bR) \cong H_2(BG; \bR)^*\cong H_2(BG; \bZ)\otimes \bR \cong \pi_1(G)\otimes \bR.\]

Following some results involving the theory of Lie groups found in \cite{DKG} the exact sequence $\pZ(G) \hookrightarrow G\twoheadrightarrow\Ad(G)$ holds for reductive $G$. Applying the fundamental group functor then implies 
\[\pi_1(\pZ(G))\to \pi_1(G) \twoheadrightarrow \pi_1 \Ad(G).\]
Now, $\pi_1\Ad(G)$ is finite implying that, after removing torsion
\[\pi_1(G)\otimes \bR \cong \pi_1(\pZ(G))\otimes \bR.\]

Characteristic classes for our bundles are constructed from the curvature tensor $F_\na \in \g \otimes \Om^2(Y)$ through contraction by a character $\chi:G\to S^1$. Notice that characters of $G$ factor through the \emph{commutator subgroup}\footnote{$[G, G] = \{aba^{-1}b^{-1}\in G:a,b\in G\}$}
(since $S^1$ is abelian) and, as a result, are actually well-defined on the quotient $G/[G, G]$. This quotient group is discretely equivalent to the center, $\pZ(G)$, of $G$ in the sense that the right side of the following exact sequence is a finite covering; 
\[\pZ(G)\hookrightarrow G \twoheadrightarrow G/[G,G].\] 
On the level of Lie algebras, however, this induces an exact sequence
\[\pZ(\g)\hookrightarrow \g \twoheadrightarrow \g/[\g,\g]\]
and hence an isomorphism $\pZ(\g) \cong \g/[\g, \g]$. Thus, the derivative of a character $d\chi:\g\to i \bR$ descends to a well-defined map $d\underline{\chi}:\pZ(\g)\to i \bR$. Also, including exponential maps to the diagram, one sees

\[\xymatrix{
&\pZ(G) \ar[r]^{\underline{\chi}} & S^1\\
\text{exp}^{-1}(1)\ar@{^{(}->}[r]&\pZ(\g)\ar[r]^{d\underline{\chi}}\ar[u]^{\text{exp}} & i \bR\ar[u]^{\text{exp}} 
}\]
where $\text{exp}^{-1}(1)$ is canonically isomorphic to $\pi_1(\pZ(G))$. 

In short, to measure the `degree' of a $G$-monopole (at least, modulo torsion) is to integrate a geometrically relevant differential form along surfaces in $S^1\times \Si$. This form should be analogous to the first Chern class from complex geometry.

With this in mind, given a singular $G$ monopole $(P, \na, \Phi)$ on $Y$, i.e. a solution to 
\[F_\na =  i C \cdot \om_\Sigma+*d_\na\Phi\]
one seeks to develop 

\subsection{The Chern-form of a monopole}

The curvature tensor $F_\na$ is given as a section of $\Om^2(\ad(P)) = \ad(P)\otimes \bigwedge^2T^* Y$. In order to obtain a first Chern form (i.e. an element of $H^2(Y, \bC)$), one must `trace-out' the Lie algebra portion of this curvature to obtain a gauge-invariant section in $\Om^2(Y)$. The degree is then measured as an integral of this form over $Y$. More concretely, to a basis $\{e_i\}_{i=1}^k$ of characters for $G$, one obtains Chern forms $\{\om_i\}$ and thus degree maps $\d_i:H_2(Y) \to \bR$ which can be adjusted to take integer values as usual.

\subsection{Groups, representations and characters of importance}

The characters of geometric relevance here are;
\begin{enumerate}
\item $\chi \in X^*(G)$ any character of $G$. This is used to determine the \emph{degree} of a monopole and is analogous to complex vector bundles when $\chi = \det$ (the only non-trivial character of $\GL_n$ whose derivative at the identity is the usual $\tr:M_n\to \bC$)
 
\item $\chi = |\Ad_L^\u|\in X^*(L)$ the unique character of $L$ (the Levi-subgroup of a maximal parabolic subgroup $H$ of $G$) given as the top exterior power of the adjoint representation of $L$ on $\u$ (the corresponding unipotent sub Lie algebra of $\h$). One important property about this character that is worth mentioning is that the center $\mathcal{Z}(G)$ of $G$ lies in the kernel of this adjoint representation so that the constant scalar portion of our curvature tensor does not affect the eventual 2-form. This will be used to measure the stability of a monopole. 
\end{enumerate}

\begin{remark}
In analogy with vector subbundles one is concerned with a maximal parabolic subgroup $H\leq G$ along with a corresponding Lie algebra decomposition 
\[\g = \mat{\mathfrak{l}_1 & \u\\ \g/\h & \mathfrak{l}_2}\]
where $\h = \mathfrak{l}_1\oplus \mathfrak{l}_2 \oplus u$ is according to the Levi decomposition of $H$
\[L\hookrightarrow H \twoheadrightarrow U.\]

\end{remark}

\begin{definition}
The \emph{Chern-form} associated to a character $\chi\in X^*(G)$ of a $G$ bundle $P$ is defined as 
\[c_1^\chi(P,\na,\Phi):= \frac{i}{2\pi}\tr^\chi(F_\na)\in \Om^2(Y)\]
where $\tr^\chi = d\chi(0)$ (or also $\chi_*$) is the derivative of $\chi$ at the identity. 
\end{definition}

Then, given a monopole with singularities at $\vec{t}$
\begin{definition}
 For a character $\chi\in X_*(G)$, the \emph{$(\chi, \vec{t})$-degree}, of a singular $G$ monopole $(P, \na, \Phi)$ 
 is the integral of the Chern-form
\[\d^\chi(P,\na, \Phi) := \frac{1}{\tau} \int_Y c_1^\chi(P,\na, \Phi)\wedge dt.\]

\noindent\textbf{Note:} Geometrically, this represents the average (along $S^1$) of the usual $\chi$-degrees along each holomorphic slice $P_{\{t\}\times\Si}$. Note that the degree of a bundle can be evaluated on any two-cycle of $Y$ (i.e. $H_2(S^1\times \Si\backslash \{p_i\}_{i=1}^N)$ is large), but that a particular choice has been made here (namely, a weighted sum over all 2-cells in the deformation retraction of $Y$ as a 2-complex). 
\end{definition}

\subsection{Integration on $S^1\times \Sigma$}

For the purpose of integration, write $Y_\e := Y \backslash\bigcup_{j=1}^N D_\e(p_j)$ to denote a closed subspace of $Y$. This $Y_\e$ limits topologically $Y$ as a nested family of closed subspaces so that integration on $Y$ is the limit (as $\e$ tends to 0) of integration on $Y_\e$.

Stokes' theorem will be of use as 
\begin{equation}\label{stokes}
\partial \left([t- \e, t + \e]\times \Sigma\backslash D_{\e/2}(p_j) \right)= \Sigma_{+} - \Sigma_{-} - S^2_{\e/2}(p_j)
\end{equation}
where $\Sigma_\pm$ denotes the surface $\{t\pm \e\}\times \Sigma$ upon restriction to times $t \pm \e$. Also, even more handy will be the fact that
\[\partial(S^1\times \overline{D_\e(z_j)}\backslash B_{\e/2}(p_j)) = S^1\times \partial\overline{ D_\e(z_j)} - S^2_{\e/2}\]
corresponding to a cylindrical neighbourhood of radius $\e$ about $z_j$

Given a character $\chi\in X^*(G)$ of $G$, define the real valued function $f^\chi: S^1\backslash \{t_1,\dots, t_N\}\to \mathbb{R}$ as \[f^\chi(t) =\frac{i}{2\pi } \int_{\{t\}\times \Sigma} c_1^\chi(P,\na, \Phi).\]

It is clear (from standard theory of Chern classes) that $f^\chi$ is an integer valued function. Furthermore,

\begin{lemma}\label{discrete} 
The function $f^\chi_t$ defined above is an integer-valued, piecewise constant function on $S^1\backslash \{t_i\}_{i=1}^N$ satisfying that for all sufficiently small $\e>0$ and singular time $t = t_j$ (for some $j$)
\[f^\chi_{t+\e}(P, \na) = f^\chi_{t-\e}(P, \na) + (\chi\circ \mu_j)_* .\]

If no singular time occurs on the interval $[t, t']$, then \[f^\chi_{t'}(P, \na) = f^\chi_t(P, \na) \] so that the discontinuities of $f^\chi_t$ occur only at the singular times. 

\end{lemma}

\begin{proof}

That $f^\chi$ is integer-valued follows directly from the fact that the Chern-form, upon restriction to $\{t\}\times \Sigma$, is an integer cohomology class. Piecewise constancy follows from the fact that the scattering map $\rho_{t,t'}$ for times $t_i<t<t'<t_{i+1}$ between singularities defines an isomorphism $P_t\cong P_{t'}$. Thus, $c_1^\chi(P_t, \na, \Phi) = c_1^\chi(P_{t'}, \na, \Phi)$ and certainly then $f^\chi_t = f^\chi_{t'}$. 

Now, on the level of homology in $Y =(S^1\times \Sigma)\backslash \{p_1,\dots, p_N\}$ where for any non-singular time $t$, $\Sigma_t:= \{t\}\times \Sigma \in H_2(Y)$ represents the fundamental homology class for the subcurve $\{t\}\times \Sigma\subset Y$. Thus, with respect to the orientations prescribed by signature in Equation \eqref{stokes} and Stokes' theorem 
 
\[f_{t+\e}^\chi (\xi) :=\int_{\Sigma_{t+\e}}\xi  = f_{t-\e}^\chi(\xi) +\int_{S^2_\e}\xi + \int_Y d\xi \]
for any $\xi\in H^2(Y)$. Here, $\xi =c_1^\chi(F_\na) = \tr^\chi F_\na$ so, making use of the Bianchi identity (that $d_\na F_\na = 0$) and that $[g, g] \leq \ker \tr^\chi$,
\begin{align*}
d\xi &= d\circ \tr^\chi F_\na = \tr^\chi \circ dF_\na = \tr^\chi( d_\na F_\na - [\na, F_\na])= - \tr^\chi[\na, F_\na] = 0.
\end{align*}
Thus far, this demonstrates that,
\[f^\chi_{t+\e}(P, \na) = f^\chi_{t-\e}(P, \na) + \int_{S^2_{\e/2}}\tr^\chi(F_\na) .\]

It remains to evaluate $\frac{1}{2\pi} \int_{S^2_{\e}}\tr^\chi(F_\na)$ which is immediately seen to be $(\chi\circ \mu)_* $
since $\chi$ defines an associated line bundle for the $T$-bundle given by $\mu$ so the computation follows from the asymptotic form of the curvature tensor about $p$. 
\end{proof}

Lemma \ref{discrete} breaks down the \emph{$\chi$-degree} of a monopole $\d^\chi(P, \na, \Phi)$ into the integral of this piecewise constant function $f^\chi_t$ as

\begin{cor}
The $\chi$-degree of $(P,\na, \Phi)$ reduces to discrete inputs  and evaluates as 
\[ \d^\chi(P, \na, \phi) = \chi_*\circ C \cdot \Vol_\Si + \frac{1}{\tau}\sum_{j=1}^N (\tau- t_j) (\chi\circ \mu_j)_*.\]

\end{cor}
\begin{proof} Recall that $\d^\chi(P, \na, \Phi) = \int_Y c_1^\chi(P, \na, \Phi)\wedge dt = \int_{S^1\backslash\{t_i\}_{i=1}^N} f^\chi_tdt\\
$ which can now be manipulated as follows
\begin{align*}
\int_{S^1\backslash\{t_i\}_{i=1}^N} f^\chi_tdt&= \sum_{i=0}^N (t_{i+1}-t_i) f^\chi_{t_i^*}\\
&= \sum_{i=0}^N (t_{i+1}-t_i) \left( f^\chi_0  +  \sum_{j=1}^i \tr^\chi (\mu_j) \right) \\
&=   \chi_*\circ C\cdot \tau\cdot \Vol_\Si  +\sum_{i=0}^N (t_{i+1}-t_i)  \sum_{j=1}^i \tr^\chi (\mu_j)\\
&=   \chi_*\circ C\cdot \tau\cdot \Vol_\Si  +\sum_{j=0}^N (\tau-t_j)  \tr^\chi (\mu_j)\\
\end{align*}
where $t_i^*\in (t_i, t_{i+1})$ is any point inside the $i^{th}$ singular interval.

\end{proof}

\section{Stability theory of monopoles and pairs}

The following definition is inspired by and consistent with Ramanathan's definition \cite{R} for stability of a holomorphic principal $G^c$-bundle over a Riemann surface. 

\begin{definition}
 A holomorphic structure $(P^c, \na^{0,1}, \na^c)$ is \emph{stable} if for every $H$-invariant holomorphic reduction $P_H\leq P^c$ where $H\leq G^c$ is a maximal parabolic subgroup, one has
\[\d^\chi(P_H) <0 \]
where $\chi = \det\circ \Ad_L^{\u}$ is the unique character of $L$ (from the Levi-decomposition $H = L\ltimes U$) whose derivative is the sum of the roots of $U$.
\end{definition}

Before proceeding with any stability results for the objects of interest here, it will be necessary to revisit a result of \cite[Proposition~2.3.1]{LT}. The following Lemma has been adapted and re-expressed in the language of principal bundles. 

\begin{lemma}\label{LTproof}
Hermitian-Einstein $G$-bundles over $\Sigma$ are polystable. 
\end{lemma}
\begin{proof}

Suppose that a Hermitian-Einstein $G$-bundle $(P, \na)$ admits a holomorphic reduction $P_H\ss P$ to a maximal parabolic subgroup $H\leq G$. The decomposition of $\g$ induced by $\h = \mathfrak{l}_1\oplus\mathfrak{l}_2\oplus \u$ allows us to decompose the connection form $\om$ (in a unitary gauge) of $\na$ into 
\[\om  = \om_1+\om_2 + \pF^*+\pF\in \g\otimes \Om^1(\Sigma) \]
where $\g =  \mathfrak{l}_1\oplus \mathfrak{l}_2 \oplus \u \oplus \g/\h$.
\begin{equation}\label{fundamentalform}
\pF = \pF(\na, H):= \na|_{TP_H} - \na_H \in \pA^{1,0}(\g/\h)
\end{equation}
is referred to as the \emph{second fundamental form} of $\na$ and visualized matrically as \[\pF =\mat{0& 0\\ f& 0} .\]

Having this expression for the connection form, the curvature is then decomposed similarly according to $\g = \mathfrak{l}_1\oplus \mathfrak{l}_2 \oplus \u \oplus \g/\h$ as

\[\Om_P = d\om_P + \om_P \wedge\om_P = \underbrace{\Om_{L_1}  +\Om_{L_2}+ \pF\wedge \pF^*}_{\in (\mathfrak{l}_1\oplus \mathfrak{l}_2)\otimes \Omega^2(\Si)}  +  \bigstar\]
where $\bigstar$ denotes all terms in $\u \oplus \g/\h$ will be neglected since characters are evaluated on maximal tori.
Thus, upon projection to $\mathfrak{l} =\mathfrak{l}_1\oplus \mathfrak{l}_2$, this is simply expressed  
\[\pi_L\circ \Om_P = \Om_{L} + \pF\wedge \pF^*\] 
which reads globally as 
\[F_{\pi_L(\na)} = \pi_L \circ F_\na- \pF\wedge \pF^*. \]

The Hermitian-Einstein condition on $F_\na$ allows us to write $F_\na = i C  \cdot \om_\Sigma$ and evaluation of the character $\chi = \Ad^\u_H$ on $H$ and will be denoted accordingly as \[\tr^\chi:=  d\chi: \mathfrak{t} \to \bC.\]

The Chern-form, $c_1^\chi(F_{\pi_L(\na)})$, associated to $\chi$ is defined by the map 
\begin{align*}
c_1^\chi:&\mathfrak{l} \otimes \Om^2(\Sigma) \to \bC\otimes \Om^2(\Sigma)\\
&F_{\pi_L(\na)}\mapsto  \tfrac{i}{2\pi}\tr^\chi(F_{\pi_L(\na)})
\end{align*}
and so
\begin{align*}
\tfrac{i}{2\pi}\tr^\chi(F_{\pi_L(\na)}) &= \tfrac{i}{2\pi}\tr^\chi(F_{\na} - \pF\wedge \pF^*)\\
&=\tfrac{i}{2\pi} \tr^\chi(iC)\om_\Sigma  -\tfrac{i}{2\pi}\tr^\chi (\pF\wedge \pF^*)\\
&=- \rk(G) ||\pF||_\chi^2\cdot \om_\Sigma
\end{align*}
Note that $\tr^\chi(i C) = 0$ since the centre of the Lie algebra is contained in the kernel of the adjoint representation. So then 
\[\d(P_L(\chi)):=\int_\Sigma c_1^\chi(F_L) = - \rk(G) ||\pF||_\chi^2 \int_\Sigma \om_\Sigma = - \rk(G) ||\pF||_\chi^2\cdot \text{Vol}_\Sigma \leq0\] 
with equality if and only if $\pF = 0$ which, furthermore, implies the existence of a reduction to the Levi subgroup of $H$. 
\end{proof}

\begin{proposition}\label{negative}
The holomorphic structure obtained from a singular $G$-monopole is stable. In particular, one finds that
\[\d^{|\Ad_L^\u|}(P_H, \na, \phi)  = -\int_Y ||\pF||^2_\chi \cdot \om_\Si\wedge dt\]
for any holomorphic reduction $P_H$ to a maximal parabolic subgroup $H$.
\end{proposition}

\begin{proof} Let $(P, \na, \Phi)$ be a singular $G$ monopole and $H\leq G^c$ a maximal parabolic subgroup of $G^c$ such that $P^c = P\times_GG^c$ admits a holomorphic reduction to $P_H$. Then, via $U\hookrightarrow H \stackrel{\pi}{\twoheadrightarrow} L$, $P_H$ projects to an $L$-bundle $P_L = \pi_L \circ P_H$. On the level of  adjoint bundles, with Levi-subalgebra $\mathfrak{l} \leq \h$ as in the proof in Lemma \ref{LTproof}, its curvature satisfies the following relation between the total curvature and its \emph{second fundamental form} $\pF$ 
\[F_{\pi_L(\na)} =  \pi_L\circ F_\na - \pF\wedge \pF^*.\]

So, for the character $\chi = |\Ad_L^\u|$ of $L$, by definition
\begin{align*}
\delta^\chi(P_L) = \lim_{\e\to 0} \int_{Y_\e} c_1^\chi(F_{\pi_L(\na)} ) \wedge dt.\\
\end{align*}
Upon substituting the HEB equation \eqref{HEB} for $F_\na$, this evaluates as
\begin{align*}
\lim_{\e\to 0}\frac{i}{2\pi}&\int_{Y_{\e/2}}\tr^\chi(i C\cdot \om_\Si + *d_\na \Phi - \pF\wedge \pF^*) \wedge dt\\
&=- \int_Y ||\pF||^2_\chi \cdot \om_\Si\wedge dt +\lim_{\e\to 0} \frac{i}{2\pi}\int_{Y_\e} \partial_t \Phi^\chi dt \wedge \om_\Si\\
&<\lim_{\e\to 0} \frac{i}{2\pi}\int_{Y_\e} \partial_t \Phi^\chi dt \wedge \om_\Si\\
\end{align*}
since $\pZ(\g_\bC)\subset \ker d\chi$ (implying that $\tr^\chi(C) = 0$) and, although non-constant, $||\pF||^2_\chi$ is strictly positive (when our monopole is irreducible). %Now it remains to demonstrate that the remaining term to vanishes.

Notice immediately that the remaining term reduces to 
\[ \lim_{\e\to 0}\frac{1}{2\pi} \sum_{j=1}^N \int_{S^1\times \overline{D_\e(z_j)} \backslash B_{\e/2}(p_j) }\partial_ti\Phi^\chi\cdot \om_\Si\wedge dt \]
because away from any nonsingular circle ($S^1\times \{z_j\}$) this amounts to 
\[\int_{S^1} \partial_t i \Phi^\chi dt = 0\]
being the integral of the derivative over a closed interval.

Now, writing $\partial_t i \Phi^\chi \om_\Si\wedge dt = d\left(i \Phi^\chi \om_\Si\right)$ as an exact form and by Stokes' theorem, each
\begin{align*}
\int_{S^1\times \overline{D_\e(z_j)} \backslash B_{\e/2}(p_j) }\partial_ti\Phi^\chi\cdot \om_\Si\wedge dt
&= \int_{S^1\times S^1_{\e}} i \Phi^\chi\cdot \om_\Si - \int_{S^2_{\e/2}} i \Phi^\chi\cdot\om_\Si\\
\end{align*}

The first term here vanishes in the limit as $\e\to 0$ and the second term is reinterpreted in a different coordinate system. Currently, there are two local coordinate systems under consideration. Namely, the connection and Higgs field have been expressed in terms of the spherical coordinates $\{dR, d\theta, d\psi\}$ whereas the form of integration is in terms of `holomorphic-Euclidean' coordinates $\{dz, d\overline{z}, dt\}$. A happy medium for choice of coordinates here will be to choose a cylinder inscribed in the $\e/2$-ball whose dimensions are chosen to be radius $\e/2\sqrt{2}$ and height $\e/\sqrt{2}$ (These are homotopy equivalent in $Y$ and hence have the same values upon integration). Recognizing the change in domain of integration to a cylinder, the second term is then seen to be bounded above by $\sup_{C_\e}(i \Phi^\chi) \cdot 2\cdot \Vol_{D_\e}$ which is $\pO(\e^2)$ according to the volume of the caps on the cylinder and thus limits to zero. That is, 
\[\lim_{\e\to 0}\frac{1}{2\pi} \sum_{j=1}^N \int_{S^1\times \overline{D_\e(z_j)} \backslash B_{\e/2}(p_j) }\partial_ti\Phi^\chi\cdot \om_\Si\wedge dt = 0.\]
\end{proof}

\subsection{$\vec{t}$-degree and stability of a meromorphic pair}

A sensible approach (as taken in \cite{CH}) to define the proper notion of stability for the algebraic data contained in a bundle pair $(\pP, \rho)$ examines the average (in $S^1$) degree of a monopole $(P, \na, \Phi)$ defined over $Y$ and manipulates this until it can be computed using only the information contained in the image $(\pP, \rho) = \pH(P, \na, \Phi)$. Results of Corollary \ref{discrete} and Proposition \ref{negative} do exactly this which allows for the following definitions without justification.

\begin{definition} Let $(\pP, \rho)$ be a meromorphic pair
\begin{enumerate}

\item The \emph{$(\chi,\vec{t})$-degree} of $(\pP, \rho)$ is defined as 
\[\d^\chi_{\vec{t}}(\pP, \rho)= \sum_{i=0}^N (t_{i+1}-t_i)\left(\d^\chi(\pP) + \sum_{j=1}^i (\chi\circ \mu_j)_*(0)\right)\]

where $\d^\chi(\pP)$ is the degree of the complex line bundle $\pP(\chi) := \pP \times_{\chi_*}\bC$

\item $(\pP, \rho)$ is \emph{$\vec{t}$-stable} if for every $\rho$-invariant holomorphic reduction to $\pP_H\subset \pP$ where $H\leq G^c$ is a maximal parabolic subgroup of $G^c$ one has 
\[\d^{|\Ad_L^\u|}(\pP_H, \rho)<0.\]

Note that $|\Ad_L^\u|$ (as a character of $H$) is the determinant of the adjoint representation of $L$ on $\u$, where $H = L\ltimes U$ is its Levi-decomposition and $\u = \Lie(U)$.
\end{enumerate}
\end{definition}

Adopting notation from the space of meromorphic pairs, the \emph{moduli space of $\vec{t}$-stable meromorphic bundle pairs} over $\Sigma$ of singular type $\textbf{K} = \{(\mu_j, z_j)\}_{j=1}^N $ will be denoted by
\[\pM_{\vec{t}s}(\Sigma, \textbf{K}).\]

Thus define that a holomorphic structure is $\vec{t}$-stable if its associated meromorphic pair is. It has been shown (through discretizing the integration - Lemma \ref{discrete}) that the holomorphic structure associated to an irreducible singular monopole is $\vec{t}$-stable. In more appropriate terminology, that is;

\begin{proposition}\label{stabprop}
If $(P,\na, \Phi) \in \pM_{k_0}^{irr}(G, S^1\times \Sigma, \{(p_i, \mu_i)\}_{i=1}^N)$ then its image under $\pH$ is $\vec{t}$-stable.
\end{proposition}

\begin{proof}
Everything for this proof has already been set up and only requires a small argument. Suppose $(\pP, \rho) = \pH(P,\na, \Phi)$ and let  $H$ be a maximal parabolic subgroup of $G^c$ corresponding to a holomorphic, $\rho$-invariant reduction $\pP_H$ of $\pP$. That $\d^{|\Ad_L^\u|}_{\vec{t}}(\pP_L)$ is negative has already been verified and is the result of Proposition \ref{negative}. 
\end{proof}

\section{The correspondence}

Now that the objects of interest are well-defined and the stability theory has been taken care of, this section is focused solely on the proof of the bijective correspondence theorem stated below. The surjectivity of $\pH$ (defined in the previous chapter) is quite analytic and heavily relies on the proof found in \cite{CH}. The injectivity of $\pH$ also follows their recipe but depends more on the theory of induced connections on associated principal bundles (fully developed for this application in the author's thesis \cite{BHSD}). %The formalities and explicit calculations of the surjectivity proof and theory of connections in associated fibre bundles are fully . 

\subsection{Equivalence between meromorphic pairs and singular monopoles}

\begin{theorem}\label{Main}
If $\{p_i\}_{i=1}^N$ is a finite subset of $S^1\times \Sigma$ which project to $N$ distinct points on $\Sigma$ then the map 
\[\pH:\pM^{irr}_{k_0}(G,S^1\times \Sigma, \{p_i,\mu_i\}_{i=1}^N)\to \pM_{\vec{t}s}(\Sigma, k_0, \mathbf{K})\]
\[(P,\na, \Phi)\mapsto (P_0, \rho_{0,\tau})\]
is a bijection. \footnote{The statement when irreducibility is removed is between poly-stable pairs.}
\end{theorem}
 
The proof demonstrated throughout the following two propositions \ref{P1} for surjectivity and \ref{Injectivity} for injectivity. Notice that this Theorem \ref{Main} is a condensed version of the main Theorem \ref{main} stated in the introduction as the reader is now assumed to be familiar with the objects at hand. 

\begin{comment}
 To get to the heart of the proof, let us first recall the objects at hand. The moduli space $\pM^{irr}_{k_0}(S^1\times \Sigma, \{p_i,\mu_i\}_{i=1}^N)$ (on the left) represents equivalences classes of triples $(P, \na, \Phi)$ satisfying Equation \eqref{HEB} and having $\mu$-Dirac singularities of type $\mu_i$ at each $p_i\in S^1\times \Sigma$ and topological type $k_0$ along $\{t_0\}\times \Si$. For notational simplicity alone, it is assumed that within the set $\{p_i = (t_i, z_i)\}_{i=1}^N$, none of the $t_i$'s or $z_i$'s coincide. Each triplet $(P, \na, \Phi)$ consists of a principal $G$-bundle on $S^1\times \Sigma\backslash\{p_1, \dots, p_N\}$, a connection $\na \in \Omega^1(\ad(P))$ and a meromorphic Higgs field $\Phi\in \pM(\ad(P))$.

The moduli space $\pM_{\vec{t}s}(\Sigma,k_0, \mathbf{K})$ (on the right) consists of equivalence classes of $\vec{t}$-stable pairs $(P, \rho)$ where $P$ is a holomorphic principal $G$-bundle on $\Sigma$ of topological type $k_0$ and $\rho\in \pM(\Aut(P))$ is a meromorphic automorphism with singularities prescribed by the $(\mu_i, z_i)\in \mathbf{K}$.
\end{comment}

\begin{proposition}\label{P1}
For any $\vec{t}$-stable pair $(\pP, \rho)$ on $\Sigma$ of type $\mathbf{K} = ((\mu_1, z_1), \dots, (\mu_N, z_N))$ with singular time data $0<t_1\leq t_2\leq \dots \leq t_n <\tau$, there is a singular $G$-monopole on $S^1\times \Sigma$ with Dirac singularities of weight $\mu_j$ at $p_j= (t_j, z_j)$ for which $\pH(P, \na, \phi) = (\pP, \rho)$. 
\end{proposition}

\begin{proof}
Upon choosing a faithful unitary representation and so embedding into the $\GL_n$ case where this result has been proven in \cite{CH}. Since stability is tautologically preserved through the representation, it suffices to state the steps involved in the language of principle bundles and point out the places where additional arguments are needed.

The four main steps of this proof are as follows
\begin{itemize}
\item[(i)]
$\rho$ is used to extend $\pP$ to a bundle $P$ on $Y := (S^1\times \Sigma) \backslash \{p_1,\dots, p_N\}$ having the correct twisting around spheres about the $p_j$'s and a holomorphic structure. Thus it will be holomorphic on all $\Sigma_t$ and will lift to a holomorphic bundle $\bar{P}$ on the (open) complex manifold $X = S^1\times Y$ (a subset of $\overline{X} = S^1\times S^1\times \Sigma$). Furthermore, $\bar{P}$ is invariant under the action of $S^1$ on the left-most factor. 

\noindent\rule{10cm}{0.4pt}

More concretely, all of these ideas fit into the following diagram 
 \[
\xymatrix{
\bar{P}:=\pi_2^*(P)\ar[d]\ar[r]&P \ar[d]& \pi^*(\pP)\ar[l]_{\tilde{q}}\ar[r] \ar[d]& \pP\ar[d]\\
X\ar[r]^{\pi_2}&Y& \tilde{Y}\ar[l]_q\ar[r]^\pi & \Sigma
}
\]

where \[\tilde{Y} = \left((-\tau, \tau)\times \Sigma\right)\backslash \cup_j\left((-\tau, t_j-\tau)\cup (t_j, \tau)\right)\times \{z_j\}\]
with $\pi:\tilde{Y}\to \Sigma$ as the natural projection and $q:\tilde{Y} \to Y$ is generically a double cover defined by the identification $(t, z) \sim (t+\tau, z)$. Next, $P: = \tilde{q}\left(\pi^*(\pP)\right)$ is given by the equivalence relation $(t,z, v)\sim (t+\tau, z, \rho(z) v)$ with $t\in (-\tau, 0), z\in \Sigma$ and $v\in \pP_z$. Finally, $P$ is trivially lifted to an $S^1$-invariant, holomorphic $\bar{P}:= \pi^*_2(P)$ via the canonical projection $\pi_2:X\to Y$ with $X = S^1\times Y$. 

\noindent\rule{10cm}{0.4pt}

\item[(ii)]
Since $\bar{P}$ has a holomorphic structure, for any Hermitian metric (i.e. any reduction of $\bar{P}$ to $G$, i.e. any section of $\bar{P}(G^c/G)$), there is a unique metric connection\footnote{this is the Chern connection in the case of a $U_n\ss \GL_n(\bC)$ gauge.} which is compatible with the holomorphic structure. Such a Hermitian metric on $\bar{P}$ is chosen so that the induced connection around the $j^{th}$ singularity is that of a $\mu_j$-Dirac monopole.% of weight $\mu_j$. 

\noindent\rule{10cm}{0.4pt}

The appropriate Hermitian metric is constructed (via a partition of unity) on the open cover of $Y$
\begin{align*}
U_0 &= \left((-2\e, t_N + 2\e)\times \Sigma\right) \backslash \left(\cup_j (t_j-\e, t_N+ 2\e)\times C_j\right)\\
U_{N+1} &= (t_N+\e, \tau-\e)\times \Sigma\\
U_{j-} &=  \left((t_j - 2\e, t_j+2\e)\times D_j \right)\backslash\left( (t_j, t_j+2\e)\times \{z_j\} \right)\\
U_{j+} &= \left((t_j - 2\e, t_j+2\e)\times D_j \right)\backslash\left( (t_j+2\e, t_j)\times \{z_j\} \right)
\end{align*}
where $D_i$ for $i = 0, 1, \dots, N$ are sufficiently small, disjoint open disks about each $p_i = (t_i, z_i)$ save for $D_0$ which is where the twisted curvature for the initial degree $k_0$ is concentrated. $C_i$ for $i = 1, \dots, N$ are another family of disks about each $z_i$ which are properly contained in the $D_i$'s and $\e>0$ is chosen so that 
\[4\e <\min(t_1, t_2-t_1,\dots, t_N-t_{N-1}, \tau-t_N).\]
 On this cover, the transition functions are specified as
\[\vf_{0,j-} = g_j, \, \vf_{j-,j+} = \mu_j, \, \vf_{0,j+} = g_j\cdot \mu_j, \, \vf_{j+, N+1} = h_j \]
and
 \[\vf_{0, N+1} = \begin{cases}\rho^{-1} , \, t\in (t_N+\e, t_N + 2\e)\\ 1, \, t\in (\tau-2\e, \tau-\e)\end{cases}. \]
 
Now, the bundle and its transition functions reflect those of $\mu$-Dirac monopoles on $U_{j\pm}$. Choose the hermitian metrics $\mu_j(R- t)$ on $U_{j-}$ and $\mu_j(R+t)$ on $U_{j+}$. These are compatible with each other under change of basis and are patched together, along with the metric lifted from $P$ on $U_0, U_{N+1}$ by partition of unity.

This metric $k$ is then  lifted to a metric $\bar{k}$ on $\overline{P}$ subject to the following properties:

\begin{lemma}
The pair $(\overline{P}, \bar{k})$ above satisfies
\begin{itemize}
\item[(a)] $\bar{P}$ is invariant under the $S^1$ action on the left factor of $X$%; this action complexifies to action of $\bC^*$ over $S^1\times S^1\times (\Sigma \backslash\{z_1,\dots, z_N\})$ with real element $T$ in $\bR\sseq \bC^*$ corresponding to $\rho$.
\item[(b)] $\bar{k}$ is $S^1$ invariant. 
\item[(c)] In neighbourhoods of inverse image of the $p_j$'s the pair $(\bar{P},\bar{k})$ corresponds to an $S^1$-invariant instanton of charge specified by $\mu_j$ . 
\item[(d)] $(\bar{P}, \bar{k})$ satisfies a bound $|\Lambda F_{\bar{k}}|\leq c <\infty$
\end{itemize}
\end{lemma}
\noindent\rule{10cm}{0.4pt}

\item[(iii)]
This metric serves as an initial metric for the heat flow of Simpson's paper \cite{S}. Taking the limit as time tends to infinity produces a principal-HE connection on $\bar{P}$ which is invariant under the $S^1$ action and so, descends to a bundle over $Y$. This will be our singular $G$-monopole, however one further analytic technicality remains. Note, in this situation we are considering a representation of a lie group embedded in $\GL_n(\bC)$ and only difference between this and the proof found in \cite{CH} is that here we must demonstrate that the heat flow remains within the desired subspace. 

\noindent\rule{10cm}{0.4pt}

Using $\bar{k}$ in $G^c/G$ from above as the starting point for Simpson's heat flow
\begin{equation}\label{heateqn}
\begin{split}
H^{-1} \frac{dH}{du} &= -i \Lambda F_H^\perp\\ H_0 &= \bar{k}
\end{split}
\end{equation}

This equation remains valid in $G^c/G$ as the left and right hand side both take values in $i\cdot \g$. 

The asymptotic behaviour of \eqref{heateqn} is governed by the following:

\begin{theorem}[Simpson \cite{S}, Theorem 1]\label{Simpson}
Let $(X,\omega)$ satisfy conditions in Lemma \ref{4.6} and suppose $E$ is an $S^1$-invariant bundle on $X$ with $S^1$-invariant metric $K$ satisfying that $\sup|\Lambda F_K|<c$. If $E$ is stable in the sense that it arises from a stable pair on $\Sigma$, then there is an $S^1$-invariant metric $H$ with $\det(H) = \det(K)$, $H$ and $K$ mutually bounded $\bar{\partial}(K^{-1}H) \in L^2$ and such that $\Lambda F_H^\perp = 0$. Additionally, \cite{CH}, if $R$ is the geodesic distance to one of the singularities, $R\cdot d(K^{-1}H)$ is bounded by a constant. 
\end{theorem}

Observe that our notion of stability generalizes the notion provided in \cite{CH} which coincides with Simpson's.

\begin{lemma}\label{4.6}
Our manifold $X = S^1\times ((S^1\times \Sigma)\backslash \{p_1,\dots, p_N\})$ satisfies the three necessary conditions for Simpson's Theorem
\begin{enumerate}
\item $X$ is K\"{a}hler and of finite volume;
\item There exists a $\geq 0$ exhaustion function with bounded Laplacian on $X$;
\item There is an increasing $a:[0, \infty)\to [0,\infty)$ such that $a(0) = 0$ and $a(x) = x$ for all $x>1$ so that if $f$ is a bounded positive function on $X$ with $\Delta(f)\leq B$, then 
\[\sup_X|f| \leq C(B)a\left(\int_X |f|\right)\]
and furthermore, if $\Delta(f)\leq 0 $ then $\Delta(f) = 0$. 
\end{enumerate}
\end{lemma}
\begin{proof} See \cite{CH} and \cite{MP}.
\end{proof}
\noindent\rule{10cm}{0.4pt}

\item[(iv)]
Simpson's theorem does not immediately provide the necessary regularity at the singular points. To see they are indeed of Dirac type (in the limit), the proof is finished by lifting locally on 3-balls using the Hopf map\footnote{Our Hopf map here is a restriction, to $B_4\sseq \bC^2$, of the well known \[\pi:\bC^2\to \bR^3;\,(z,w)\mapsto \left(z\bar{w} + w\bar{z}, i(z\bar{w} - w\bar{z}), |z|^2 - |w|^2\right))\]} $\pi:B^4\to B^3$. 

\noindent\rule{10cm}{0.4pt}

To avoid a very lengthy rehashing of the proof provided in both \cite{CH, BHSD} an illustrative description of this step is provided here and the reader is referred to either reference for all technicalities. 

Heuristically, to ensure regularity of $H_\infty$ from the previous step, restrict our initial metric, $H_0$, to a neighbourhood of $p$ (diffeomorphic to $B^3\backslash \{p\}$ and desingularize by extending the pullback of our Hopf map $\pi^*:\Om(B^3\times S^1)\to \Om(B^4)$. This process of desingularization was brought to light by Kronheimer \cite{Kr} and studied in more depth by Pauly \cite{MP}. After all differential forms of interest are pulled back and appropriately scaled, apply Simpson's heat flow (with Dirichlet boundary conditions) to the HEB equation \eqref{HEB} which is known to correspond to an $S^1$-invariant instanton equation. A new Hermitian metric is achieved in the heat-flow limit and then pushed back down to a metric describing a Dirac monopole at $p$ which, due to the imposed Dirichlet boundary condition, glues right back into the global picture. Finally, due to uniqueness of the solutions to Simpson's heat flow, this ``alternate'' solution is found to coincide with the previous and thus the previous metric satisfies the required regularity at the singularities. 
 
\noindent\rule{10cm}{0.4pt}
\end{itemize}
\end{proof}

Note the change in notation for $\tau$ below as it is no longer needed to denote circumference.

\begin{proposition}\label{Injectivity}
If two singular $G$-monopoles $(P, \na, \Phi)$ and $(P', \na', \Phi')$ yield isomorphic holomorphic data, then they are isomorphic (i.e. $\pH$ is injective). 
\end{proposition}

%\begin{remark} The main source of technical difficulty lies with the lack of a tensor product operation for principal bundles. If this were not an issue, one could imagine a natural way of inducing a monopole structure on some algebraically combined version of the two monopoles (i.e. one whose sections are the bundle maps $P\to P'$) and proceed in showing that the sections here as well as the induced Higgs field are covariantly constant with respect to the induced connection.  \end{remark}

\begin{proof}
Having the proof for vector bundles in mind (c.f. \cite{CH}), note that $\Hom_G(P, P')$ is realized as the associated $G$-fibre bundle $(P\times_B P')\times_\vf G$ where $\vf$ is the action of $G\times G$ on $G$ defined as \[(g, h)\cdot x:= g^{-1}xh = L_{g^{-1}}\circ R_h (x).\]

If $(P, \na, \phi)$ and $(P', \na', \phi')$ are singular $G$ monopoles such that \[\pH(P, \na, \phi) = (\pP,\rho)\cong (\pP', \rho') = \pH(P', \na', \phi'),\] then $\pP \cong \pP'$ are isomorphic as holomorphic principal bundles via some $G$-equivariant bundle map $\tau:\pP\to \pP'$ which furthermore satisfies $\tau\circ \rho = \rho'\circ \tau$. This holds more generally for each $\pP_t$ and $\pP_t'$ (as a result of scattering and intertwining with meromorphic data) meaning that $\tau$ aligns the invariant fibres of $\rho$ and $\rho'$ and so extends to an isomorphism $\hat{\tau}$ between $P$ and $P'$ over $S^1\times \Sigma$. This isomorphism $\hat{\tau}$ is viewed as a section of the $G$-fibre bundle $\Hom_G(P, P')$ which is equipped with the induced connection $\hat{\na} = (\na\times \na')\times _\vf \textbf{1}$, a Higgs field $\hat{\phi} = \phi'\otimes \textbf{I} - \textbf{I}\otimes \phi$ and furthermore $\hat{\tau}_*\in \ker(\hat{\na}_\Sigma^{0,1})\cap \ker(\hat{\na}_t- i \hat{\phi})$. Using the identities 
\begin{equation}\label{identities}
\begin{aligned}
\hat{\na}_\Sigma^{1,0}\hat{\na}_\Sigma^{0,1} &= (\Delta_\Sigma + i \hat{F}_\Sigma)\om
\\
(\hat{\na}_t+ i \hat{\phi})(\hat{\na}_t-i\hat{\phi}) &= \hat{\na}_t^2+ \hat{\phi}^2 - i\hat{\na}_t\hat{\phi},
\end{aligned}
\end{equation}
(performing integration by parts in a representation of $G$) one finds 
\begin{align*}
0 & = - \int_{S^1\times \Sigma} \langle \hat{\tau}_*, (\hat{\na}_t+ i \hat{\phi})(\hat{\na}_t-i\hat{\phi})\hat{\tau} + \om^{-1}\hat{\na}_\Sigma^{1,0}\hat{\na}_\Sigma^{0,1}\hat{\tau}_*\rangle d\nu\\
&=  \int_{S^1\times \Sigma}\langle \hat{\tau}_*, (-\hat{\phi}^2 - \hat{\na}^2_t - \hat{\Delta}_\Sigma) \hat{\tau}_*\rangle d\nu\\
&=  \int_{S^1\times \Sigma}|\hat{\phi}\hat{\tau}_*|^2 + |\hat{\na}_t\hat{\tau}_*|^2 + |\hat{\na}_\Sigma\hat{\tau}_*|^2d\nu.
\end{align*}
Hence, $\hat{\tau}$ is covaritantly constant and as a map $E\to E'$ it intertwines the two Higgs fields (i.e. $\hat{\phi}\circ\hat{\tau} = 0 $ is equivalent to $ \phi'\circ \hat{\tau} - \hat{\tau}\circ \phi = 0$). Therefore, the two monopoles are isomorphic.
\end{proof}

\section{Abelianization of meromorphic pairs}

For vector bundles our meromorphic pair $(\pE, \rho)$ can be transformed into an $n$-sheeted ramified cover $S_\rho$ of $\Si$ recording the spectrum of the automorphism $\rho$ and a sheaf $\pL$ which is (generically) a line bundle on the spectral cover $S_\rho$, corresponding to the eigenvectors of $\rho$.  More generally, for reductive $G^c$-fibrations, a similar process will yield pairs $(S_\rho, \mathcal{Q})$ where $S_\rho\to \Si$ is a $|W(G^c, T^c)|$-sheeted ramified cover of $\Si$ (called a \emph{cameral cover}) and $\mathcal{Q}$ is a $T^c$-bundle over $S_\rho$.

An inverse for these constructions are provided in several places throughout the literature \cite{Don, DG, H, HuMa,HuMa2, Sc, ScE} with varying levels of abstraction and difficulty.

\subsection{Spectral data associated to a bundle pair}

Consider the bundle pair $(\pE, \rho)$ where $\pE$ is a holomorphic vector bundle over a Riemann surface $\Si$ and $\rho$ is a meromorphic automorphism of $\pE$. To elaborate a bit further, one may express this as an automorphism away from $\{z_1, z_2, \dots, z_n\}$ with near $z_j$, $\rho$ may be expressed locally as \[\rho(z) = g(z)\diag\left( (z-z_j)^{k_1}, \dots ,(z-z_j)^{k_n} \right)h(z)\] (i.e. it is meromorphic in the sense that it has poles and zeros at some points). 

Analogously, a principal bundle pair $(\pP, \rho)$ will be a principal $G$-bundle over $\Si$ and $\rho\in \mathcal{M}(\Ad_P)$ a meromorphic section of $\Ad_P = P\otimes_G G$ (where $G$ acts by conjugation). The procedure developed here is referred to as the \emph{abelianization} of the bundle pair. 

%Given a meromorphic principal pair $(\mathcal{P}, \rho)$, we wish to construct a pair $(S_\rho, \mathcal{Q})$ that analogously encodes the ``eigen data''.

\subsection{The spectral information (Cameral cover)}
From (\cite{HuMa} section 6.2), given the data $(\mathcal{P}, \rho)$, the meromorphic endomorphism $\rho\in \pM(\Aut(P))$ has a notion of spectrum given by examining its orbits under conjugation by $G$ as follows:

Fix a maximal torus $T^c$ (analogous to diagonal matrices) and to each $z\in \Si$, associate to $\rho|_{P_z}$, the Weyl group orbit in $T^c$ of the closure of the $G^c$-orbit (under conjugation) of the second coordinate in the equivalence class $\rho(p_z) = [p_z, \psi(z)] = \{(g\cdot p_z, g\psi(z) g^{-1}):g\in G^c\}$. That is, 

\begin{equation}\label{Spec}
S_\rho^0:= \{(z, \alpha)\in \Sigma\times T^c: \al\in \overline{\mathcal{O}_{G^c}(\psi(z))}\cap T\}
\end{equation}
where $\pO_{G^c}(\psi(z)) = \{g\psi(z)g^{-1}:g\in G^c\}$ is the conjugacy class of $\psi(z)$ in $G^c$.

Now, at first glance, since our torus here will be $T^c \cong (\bC^*)^n$ with $n$ as the rank of $G^c$, a first natural assumption might be that a compactification should be simply given by including the points $\{0, \infty\}$ for each copy of $\bC^*$. However, in most cases, this naive approach will not yield the desired \emph{Weyl-invariant compactification}. Assuming (to be discussed below) for a second that such an invariant compactification was at our fingertips, then $S_\rho:= \overline{S_\rho^0}^{W}$ defines a (generically) $|W(G^c,T^c)|$-fold branched cover of the Riemann surface, denoted by $q: S_\rho\to \Si$ (a projective subvariety of $\Si\times \overline{T^c}^W$).
\begin{comment}

\begin{remark} Elaborating from the case $\U_n^c = \GL_n$ and relate the spectral curve for vector bundles (an $n$-fold cover over the base) to this generalization which is $n!$-fold. This map $q$ is related to the case above via a projection map onto the first coordinate of diagonal matrices. That is, $q = \text{pr}_1\circ \pi$ from above. Recall, $\pi$ is an $n$-fold cover and projection on to the first coordinate of such a cover is $(n-1)!$-sheeted. Pictorially, for vector bundles, we have

\[
\xymatrix
{
\tilde{S}_\rho\ar[rr]^{\text{pr}_1}_{(n-1)!} \ar@/^2pc/[rrr]^{q}_{n!}& &S_\rho\ar[r]^\pi_{n} & \Si \\
 }
\]
as simply
\[(\la_1(z), \la_2(z), \dots, \la_n(z))\mapsto  \la_1(z)\mapsto  z\]
and this $\tilde{S}_\rho$ here represents the generalization for $G^c$-bundles. 
\end{remark}
\end{comment}

\subsection{A maximal torus bundle on the cameral cover}
Next, with the spectral information in hand, pullback $\mathcal{P}$ via $q$ to a bundle on $S_\rho$.

Fixing some Borel subgroup $B\leq G^c$ containing $T^c$ 
, it is known (by the Lie-Kolchin Theorem) that any group element may be conjugated into $B$. However, previously, there was no canonical choice for doing so. Having now separated the different possible semi-simple components in the $G^c$-orbit of $\rho$, this lifted bundle $q^*P$ should now admit a canonical reduction to $B$.

Indeed, writing $B$ as the semi-direct product $T^c\ltimes U$,  define 
\[P_B = \{p_{(z,\al)} \in q^*P: q^*\rho(p) = [p, \al\cdot u], \text{ for some } u\in U\}\subset q^*P.\]
That is to say $P_B$ is the family for frames for which $P$ is of the form $\al\cdot u$. Then, appealing to the fact that Borel subgroups are self-normalizing (i.e. $N_G(B)= B$), one find that the condition 
\[(p, \al \cdot u) \sim (h\cdot p, h\al \cdot u h^{-1}) = (h\cdot p, \al \cdot u')\]
for some $u'\in U$ holds if and only if $h\in B$. Hence $P_B$ is a reduction of the pullback $q^* P$ over $S_\rho$ to $B$.  Furthermore, the lifted map $q^*\rho$ is naturally found as a section of the associated reduction $\Aut(P_B)  = \text{Ad}_{P_B}$.

Now, through the isomorphism $B\cong T^c\ltimes U$, which gives the exact sequence $U\hookrightarrow B\stackrel{\pi}{\twoheadrightarrow} T$, the reduced $B$-bundle $P_B $ as an element of the non-abelian sheaf cohomology group $H^1(S_\rho; B)$ naturally also defines an element $\pi \circ P_B \in H^1(S_\rho; T)$ which is denoted by $\mathcal{Q}$. This $\mathcal{Q}$ is the desired $T$-bundle over $S_\rho$ alluded to above for which we would like to consider the pair $(S_\rho, \mathcal{Q})$ as the \emph{abelianization} of $(\pP, \rho)$. 

Note that, furthermore the unipotent information $U_{\mathcal{Q}}$ is realized as the pre-image $\pi^{-1}(\mathcal{Q}) \in H^1(S_\rho, U_{\mathcal{Q}})$ where, say, at $(z, \al)\in S_\rho$ \[U_{(z,\al)} = \pi^{-1}(\al) = \{b\in B: \exists u\in U, b = \al\cdot u\}.\]

\begin{remark}
 A reversal of this procedure, at least in the generic setting, is outlined in \cite{H} section 2. By generic, one means that the logarithm of the cameral cover (so to take values in $\mathfrak{t}$ rather than $T^c$) crosses walls of the Weyl-chamber transversally and never more than one at a time. This implies that the stabilizers at branch points are isomorphic to $\bZ/2$ and there exists a choice of gauge for which $\rho$'s orbit contains an element appearing, in matrix form, as $\smat{a& 1\\0& a}\oplus\diag(\la_1, \dots, \la_{n-2})$ with distinct $\la_1, \dots, \la_{n-2}$. 
\end{remark}

In the spirit of providing a displayed result, this shows

\begin{proposition}
There is a map from the moduli space of $\vec{t}$-stable meromorphic $G$-bundle pairs $(\mathcal{P}\to \Sigma, \rho)$ with some suitably defined moduli space of pairs $(S_\rho, \mathcal{T})$ where 

\begin{itemize}
\item $S_\rho\to \Sigma$ (a cameral cover of a Riemann surface) is the Weyl-invariant compactification of the spectral curve obtained from $(\mathcal{P}, \rho)$ through the Lie group analogue of Jordan canonical form of matrices and 
\item $\mathcal{T}$ (a maximal torus bundle on $S_\rho$) is obtained as a projection to the maximal torus of the Borel reduction achieved upon pulling back $\mathcal{P}$ to the cameral cover $S_\rho$. 
\end{itemize}
\end{proposition}

\subsection{Weyl-invariant compactifications of maximal tori}
Now, as mentioned, in the standard case, when $\U_n^c = \GL_n$, one simply compactifies its maximal torus $(\bC^*)^n$ to $(\bC P^1)^n$ by the natural extension of the two point ($\{0, \infty\}$) compactification of $\bC^*$. Any point here is invariant under permutation (i.e. the Weyl-group of $\GL_n$). Notice that $\SL_n$ has the same Weyl-group as $\GL_n$, but the maximal torus is only $(n-1)$-dimensional. Of course then, since algebraic groups faithfully embed into $\GL_N$ (for some $N$), one can expect to realize the compactification of their tori as compact subvarieties of $(\bC^*)^N$. In fact, given a complex reductive Lie group $G^c$ of rank $k$, a general procedure is stated as follows; Consider maximal $T^c\subset G^c$ (isomorphic to $(\bC^*)^k$) along with the embedding $\iota:G^c\hookrightarrow \GL_N$. Compactify the torus to $\overline{T^c} \cong (\bC P^1)^k$ and find its image under $\iota$ as a $k$-dimensional subvariety in $\overline{T_{\GL_N}} \cong (\bC P^1)^N$.

\begin{example}
One can provide a sketch of some low-dimensional cases\\
1. $G^c = \SL_3(\bC)$ has rank 2. A natural choice of maximal torus is already embedded in $T_{\GL_3}$ as $\{(x,y,z)\in (\bC^*)^3:xyz = 1\}$ (may require desingularization at $\infty$). Notice immediately that certain combinations of zeros and infinities in $(\bC P^1)^3$ are not compatible with the constraint $xyz = 1$. It suffices to check the image of $(\bC P^1)^2$ in $(\bC P^1)^3$ under the map $(x,y)\mapsto (x,y,(xy)^{-1})$. Upon doing so, one finds a complex hexagon as the image of $\overline{T_{\SL_3}}$ inside of $\overline{T_{\GL_3}}\cong (\bC P^1)^3$. \\
2. $G^c = \Sp_2(\bC)$ has rank 2. A natural choice of maximal torus is embedded in $T_{\GL_4}$ as $\{(x,y,z,w):xz=1, yw=1\}$ and verifying the image of zeros and infinities through the map $(x,y) \mapsto (x,y,x^{-1}, y^{-1})$ reveals a complex quadrilateral as a codimension 2 subvariety in $(\bC P^1)^4$.
\end{example}


\begin{thebibliography}{99}

\bibitem{CH} B. Charbonneau, J. Hurtubise, \emph{Singular Hermitian-Einstein Monopoles on the Product of a Circle and a Riemann Surface}, Int. Math. Research Notices, v.2011, No. 1, pp.175-216.

\bibitem{Don} R.Y. Donagi, \emph{Spectral Covers}, MSRI series Vol. \textbf{28}, 1995.

\bibitem{DG} R.Y. Donagi, D. Gaitsgory, \emph{The gerbe of Higgs bundles},  Transform. Groups \textbf{7} (2002) 109-153.

\bibitem{Do83} S.K. Donaldson, \emph{A new proof of a theorem of Narasimhan and Seshadri}, J. Differential Geom. Vol. 18, Number 2 (1983), 269-277.

\bibitem{Do85} S.K. Donaldson, \emph{Anti self-dual Yang-Mills connections over complex algebraic surfaces and stable vector bundles}, Proc. London Math. Society, \textbf{50} (1985) 1--26.

\bibitem{Do87} S.K. Donaldson, \emph{Infinite determinants, stable bundles and curvature}, Duke Mathematical Journal, \textbf{54} (1987) 231--247.

\bibitem{DKG} J. J. Duistermaat, J. A. C. Kolk, \emph{Lie Groups}, Universitext, Springer Berlin Heidelberg ISBN: 978-3-540-15293-4 (2000). 


\bibitem{H} J.C. Hurtubise, \emph{The algebraic geometry of the Kostant-Kirillov form}, J. London Math. Soc. (2) 56 (1997) 504--518.

\bibitem{HuMa} J.C. Hurtubise and E. Markman, \emph{Rank 2 integrable systems of Prym varieties}, Adv. Theo. Math. Phys., 2 (1998), 633--695.

\bibitem{HuMa2} J.C. Hurtubise and E. Markman, \emph{Elliptic Sklyanin integrable systems for arbitrary reductive groups}, Adv. Theor. Math. Phys., 6 (2002), 873--978.

\bibitem{K} S. Kobayashi, \emph{The Differential Geometry of Complex Vector Bundles}, Princeton University Press (1987).

\bibitem{Kr} P. B. Kronheimer, \emph{Master's Thesis}, Oxford, 1986.

\bibitem{LT} M. L\"{u}bke, A. Teleman, \emph{The Kobayashi-Hitchin Correspondence}, World Scientific (1995).

\bibitem{NS} M.S. Narasimhan, C.S. Seshadri, \emph{Stable and unitary bundles on a compact Riemann surface}, Ann. of Math. \textbf{82} (1965), 540--564.

\bibitem{MP} M. Pauly, \emph{Monopole moduli spaces for compact 3-manifolds}, Mathematische Annalen 311, no. 1 (1998): 125-46.

\bibitem{R} A. Ramanathan, \emph{Stable principal bundles on a compact Riemann surface}, Math. Ann. 213 (1975), pp. 129-152.

\bibitem{Sc} R. Scognamillo, \emph{Prym-Tjurin varieties and the Hitchin map}, Math. Ann. \textbf{303}, 47-62 (1995).

\bibitem{ScE} R. Scognamillo, \emph{An elementary approach to the abelianization of the Hitchin system for arbitrary reductive groups}, Compos. Math. \textbf{110} (1998), 17-37.

\bibitem{S} C. T. Simpson, \emph{Constructing variations of Hodge structure using Yang-Mills theory and applications of uniformization}, Journal of the AMS 1, no.4 (1988): 867-918.

\bibitem{BHSD} B. H. Smith, \emph{Singular $G$-monopoles on circle bundles over a Riemann surface}, D. Phil. McGill University 2015.

\bibitem{UY86} K. Uhlenbeck, S.T. Yau, \emph{On the existence of Hermitian Yang-Mills connections in stable vector bundles} Comm. Pure Appl. Math., \textbf{39} (1986) 257--293.

\bibitem{UY89} K. Uhlenbeck, S.T. Yau, \emph{A note on our previous paper: On the existence of Hermitian Yang-Mills connections in stable vector bundles} Comm. Pure Appl. Math., \textbf{42} (1989) 703--707. 

\end{thebibliography}
\end{document}